\documentclass[oneside, a4paper,reqno]{amsart}
\usepackage{amscd,amssymb,amsmath,graphicx,verbatim,mathrsfs,xcolor}
\usepackage{wasysym}
\usepackage{hyperref}
\usepackage[TS1,OT1,T1]{fontenc}
\usepackage{lscape} 
\usepackage{fullpage} 
\usepackage[all]{xy}
\newtheorem{theorem}{Theorem}[section]
\newtheorem{lemma}[theorem]{Lemma}
\newtheorem{corollary}[theorem]{Corollary}
\newtheorem{proposition}[theorem]{Proposition}

\newtheorem*{theorem*}{Theorem}

\theoremstyle{definition}
\newtheorem{definition}[theorem]{Definition}
\newtheorem{notation}[theorem]{Notation}

\newtheorem{example}[theorem]{Example}

\theoremstyle{remark}
\newtheorem{remark}[theorem]{Remark}

\DeclareMathOperator{\Hom}{Hom}
\DeclareMathOperator{\Ext}{Ext}

\newcommand{\Bcal}{\ensuremath{\mathcal{B}}}

\newcommand{\Dcal}{\ensuremath{\mathcal{D}}}
\newcommand{\Xcal}{\ensuremath{\mathcal{X}}}
\newcommand{\Ycal}{\ensuremath{\mathcal{Y}}}
\newcommand{\Tcal}{\ensuremath{\mathcal{T}}}
\newcommand{\Fcal}{\ensuremath{\mathcal{F}}}
\newcommand{\Ccal}{\ensuremath{\mathcal{C}}}
\newcommand{\Lbb}{\ensuremath{\mathbb{L}}}

\newcommand{\Mcal}{\ensuremath{\mathcal{M}}}

\newcommand{\Wcal}{\ensuremath{\mathcal{W}}}

\newcommand{\Acal}{\ensuremath{\mathcal{A}}}

\newcommand{\Ncal}{\ensuremath{\mathcal{N}}}
\newcommand{\Hcal}{\ensuremath{\mathcal{H}}}
\newcommand{\Scal}{\ensuremath{\mathcal{S}}}

\newcommand{\Ucal}{\ensuremath{\mathcal{U}}}
\newcommand{\Vcal}{\ensuremath{\mathcal{V}}}
\newcommand{\Zcal}{\ensuremath{\mathcal{Z}}}

\newcommand{\Kcal}{\ensuremath{\mathcal{K}}}

\newcommand{\im}{\mbox{\rm{Im\,}}}

\newcommand{\Ker}[1]{\mbox{\rm{Ker}}\,#1}
\newcommand{\Coker}[1]{\mbox{\rm{Coker}}\,#1}

\newcommand{\Kbb}{\mathbb{K}}
\newcommand{\Z}{\mathbb{Z}}
\newcommand{\Zbb}{\mathbb{Z}}
\newcommand{\Rbb}{\mathbb{R}}

\numberwithin{equation}{section}

\begin{document}
\title{Definability and approximations in triangulated categories}\
\author{Rosanna Laking, Jorge Vit\'oria}
\address{Rosanna Laking,  Dipartimento di Informatica - Settore di Matematica, Universit\`a degli Studi di Verona, Strada le Grazie 15 - Ca' Vignal, I-37134 Verona, Italy} 
\email{rosannalaking@gmail.com}
\urladdr{http://www.math.uni-bonn.de/people/laking/}
\address{Jorge Vit\'oria, Department of Mathematics, City, University of London, Northampton Square, London EC1V 0HB, United Kingdom}
\email{zjorgevitoria@gmail.com}
\urladdr{https://sites.google.com/view/jorgevitoria/}
\begin{abstract}
We give criteria for subcategories of a compactly generated algebraic triangulated category to be precovering or preenveloping. These criteria are formulated in terms of closure conditions involving products, coproducts, directed homotopy colimits and further conditions involving the notion of purity. In particular, we provide sufficient closure conditions for a subcategory of a compactly generated algebraic triangulated category to be a torsion class. Finally we explore applications of the previous results to the theory of recollements.
\end{abstract}
\thanks{RL acknowledges support from the Max Planck Institute for Mathematics (Bonn). JV acknowledges support from the Engineering and Physical Sciences Research Council of the United Kingdom, grant number EP/N016505/1. The authors thank Lidia Angeleri H\"ugel, Joseph Chuang and Frederik Marks for valuable discussions.
}
\maketitle

\section{Introduction}
Given a subcategory $\Xcal$ of a given category $\Ccal$, it is natural to ask whether every object of $\Ccal$ admits morphisms from $\Xcal$ satisfying a universal property (being terminal) among all morphims from $\Xcal$.  This is the notion of $\Xcal$ being precovering and the notion of preenveloping is dual. For example, modules over a ring always admit projective precovers and injective envelopes; these are universal maps to/from a given module from/to the distinguished classes of projective/injective modules. Approximation theory studies the subcategories $\Xcal$ that can provide such approximation maps to all objects in the category. Some subcategories providing good approximations arise through the notion of torsion pairs and definability.

Torsion pairs in abelian or triangulated categories are special pairs of subcategories, one of them being precovering and the other preenveloping. In triangulated categories such pairs assume particular relevance when they `behave well' with respect to the shift functor: these torsion pairs are called t-structures or co-t-structures.  Areas where these torsion pairs play a particularly important role include silting/tilting theory (e.g. \cite{KV,KN,KY}); the study of derived equivalences in algebra and geometry (e.g. \cite{PV} and references therein); and the study of Bridgeland's stability conditions (\cite{Bri}).

Definable subcategories of a module category (\cite{CB}) have their origins in the model theory of modules and are well-known to have good approximation-theoretic properties. Indeed, they are both precovering and preenveloping (\cite{CPT}).  In fact, a subcategory is precovering if it is closed under coproducts and pure quotients (\cite{HJ}) and it is preenveloping if it is closed under products and pure subobjects (\cite{RS}). The notion of definable subcategories of (compactly generated) triangulated categories has appeared in \cite{Krause2}, and are known to be preenveloping (\cite{AMV3}), but they are generally less well-understood.

For categories of modules over a ring, a subcategory is a torsion class if and only if it is closed under coproducts, quotients and extensions (\cite{D}), and a subcategory is definable if and only if it is closed under products, pure subobjects and pure quotients (\cite{CB} and Theorem \ref{lfp precover}(1)).  In this paper we provide analogous closure conditions for subcategories of triangulated categories that yield good approximation-theoretic properties. However, in triangulated categories all monomorphisms and epimorphisms split, and so we are lacking a useful notion of subobject and quotient object.  We approach this deficit by considering the class of compactly generated algebraic triangulated categories, which include, in particular, derived categories of rings. The assumption that $\Tcal$ is compactly generated allows us to deal with \textit{pure subobjects} and \textit{pure quotient objects}.  We summarise our main results as follows.

\begin{theorem*}
Let $\Tcal$ be a compactly generated algebraic triangulated category and $\Xcal$ a subcategory of $\Tcal$. Then the following statements hold.
\begin{enumerate}
\item If $\Xcal$ is closed under coproducts and pure quotients, then $\Xcal$ is precovering. If $\Xcal$ is closed under products and pure subobjects, then $\Xcal$ is preenveloping;
\item The subcategory $\Xcal$ is definable in $\Tcal$ if and only if it is closed under pure subobjects, products and pure quotients. In particular, definable subcategories are both precovering and preenveloping.
\item If $\Xcal$ is closed under coproducts, extentions and pure quotients, then $\Xcal$ is a torsion class in $\Tcal$.
\end{enumerate}
\end{theorem*}

In order to obtain (2) we make reference to a recent characterisation of definable subcategories when $\Tcal$ is the underlying category of a strong and stable derivator (\cite{Laking}): a subcategory is definable if and only if it is closed under products, pure subobjects and directed homotopy colimits.  The assumption that $\Tcal$ is both compactly generated and algebraic means that $\Tcal$ is equivalent to $\Dcal(\Acal)$ for a small differential graded (dg, for short) category $\Acal$ (\cite{Keller}). In particular, this means that $\Tcal$ is equivalent to the homotopy category of a nice model structure and so we may apply results from the setting of strong and stable derivators (\cite{Groth}).  However, it is not clear that the notion of `closure under directed homotopy colimits' is independent of the choice of $\Acal$.  Thus the second part of the theorem above should be seen as a simplification of \cite{Laking} in the case where $\Tcal$ is algebraic, which also relieves us of the choice of derivator.

The results summarised above have direct implications in other lines of research. Indeed, for example, our techniques make explicit a reformulation (already implicit in \cite{Krause2.5}) of the Telescope Conjecture for a compactly generated triangulated category in the case where the category is also algebraic. 

The structure of the paper is as follows. Section \ref{Sec: prelim} consists of some preliminaries concerning approximation theory, compactly generated triangulated categories and derived categories of small dg categories. In Section \ref{sec: compact algebraic} we consider compactly generated algebraic triangulated categories, which can be realised both as the homotopy category of an abelian model category and the stable category of a Frobenius exact category. We recall the background that we require on these topics and then use it to prove some preliminary results on closure conditions for subcategories of such categories. In Section \ref{Sec: Approx} we use those closure conditions to obtain approximation-theoretic properties for certain subcategories. In Section \ref{Sec: tors} we turn to the topic of torsion pairs in triangulated categories and we formulate sufficient conditions for a subcategory to be a torsion class. Finally, Section \ref{Sec: Recoll} discusses applications of the previous sections to recollements and the Telescope Conjecture.\\

\textbf{Notation and conventions.} By `subcategory' of a given category, we mean a `full and strict subcategory'. Therefore, all subcategories considered are determined by the objects lying in them, and we often refer to the subcategory as the class of its objects. Given a category $\Xcal$ and a subcategory $\Scal$ of $\Xcal$, we denote by $\Scal^\perp$ the subcategory of all $X$ in $\Xcal$ such that $\Hom_\Xcal(S,X)=0$, for all $S$ in $\Scal$. If a subcategory $\Ycal$ is contained in $\Scal^\perp$ then we write $\Hom_\Xcal(\Scal,\Ycal)=0$. Given a small category $\Xcal$, we denote by $\mathsf{Mod}(\Xcal)$ the abelian category of contravariant additive functors from $\Xcal$ to the category of abelian groups $\mathsf{Mod}(\Z)$. Finally, given $\Ccal$ a subcategory of an additive category $\Xcal$, we denote by $\mathsf{Add}(\Ccal)$ the subcategory of $\Xcal$ whose objects are summands of direct sums of objects in $\Ccal$.

\section{Preliminaries}\label{Sec: prelim}
In this section we begin with some definitions and statements that we will need later.  In Subsection \ref{subsec: Approx} we briefly recall some known results about approximation theory in locally finitely presented categories. Following this, in Subsection \ref{Sec: comp gen} we define compactly generated triangulated categories and concentrate in particular on the theory of purity therein. Subsection \ref{Sec: Comp gen alg} consists of a short overview of derived categories of small dg categories. 

\subsection{Approximation theory}\label{subsec: Approx}
We begin with the following definitions from approximation theory.
\begin{definition}
A subcategory $\Xcal$ of an additive category $\Acal$ is said to be 
\begin{itemize}
\item \textbf{precovering} if for every object $A$ in $\Acal$ there is $X$ in $\Xcal$ and a morphism $f\colon X\longrightarrow A$ such that $\Hom_\Acal(X^\prime,f)$ is surjective for all $X^\prime$ in $\Xcal$. The map $f$ is then called a $\Xcal$-precover of $A$.
\item \textbf{preenveloping} if for every object $A$ in $\Acal$ there is $X$ in $\Xcal$ and a morphism $f\colon A\longrightarrow X$ such that $\Hom_\Acal(f,X^\prime)$ is surjective for all $X^\prime$ in $\Xcal$. The map $f$ is then called a $\Xcal$-preenvelope of $A$.
\item \textbf{covering} if every object $A$ in $\Acal$ admits a right minimal $\Xcal$-precover (i.e.~a $\Xcal$-precover $f\colon X\longrightarrow~A$ such that any endomorphism $g$ of $X$ with the property that $fg=f$ must be an automorphism).
\item \textbf{enveloping} if every object $A$ in $\Acal$ admits a left minimal $\Xcal$-preenvelope (i.e.~a $\Xcal$-preenvelope $f\colon A\longrightarrow X$ such that any endomorphism $g$ of $X$ with the property that $g f=f$ must be an automorphism).
\item \textbf{coreflective} if the inclusion functor of $\Xcal$ into $\Acal$ admits a right adjoint.
\item \textbf{reflective} if the inclusion functor of $\Xcal$ into $\Acal$ admits a left adjoint.
\item \textbf{bireflective} if the inclusion functor of $\Xcal$ into $\Acal$ admits both a left and a right adjoint.
\end{itemize}
\end{definition}

Applying the counit of the adjunction we have that that any coreflective subcategory is covering and, dually, the unit of the adjunction shows that any reflective subcategory is enveloping.

We will consider sufficiently nice additive categories in which certain closure conditions on a subcategory imply that it is precovering and/or preenveloping. 

\begin{definition}\label{Def: lfp}
Let $\Acal$ be a cocomplete additive category (i.e.~an additive category with all colimits).  An object $A$ in $\Acal$ is \textbf{finitely presented} if $\Hom_\Acal(A,-)$ commutes with directed colimits. The category $\Acal$ is said to be \textbf{locally finitely presented} if there is a set $\Scal$ of finitely presented objects such that every object of $\Acal$ is a directed colimit of objects in $\Scal$.
\end{definition}

As $\Acal$ is a cocomplete additive category, we can also define the  notions of pure subobject and pure quotient.  This definition will be used to express closure conditions in Theorem \ref{lfp precover}.

\begin{definition}
Let $\Acal$ be a cocomplete additive category. A morphism $f\colon X\longrightarrow Y$ is said to be a \textbf{pure monomorphism} (respectively, a \textbf{pure epimorphism}) in $\Acal$ if there is a directed system of split monomorphisms (respectively, split epimorphisms) $(f_i\colon X_i\longrightarrow Y_i)_{i\in I}$ such that $f=\varinjlim_I f_i$. An object $X$ is said to be a \textbf{pure subobject} (respectively, a \textbf{pure quotient}) of $Y$ if there is a pure monomorphism $X\longrightarrow Y$ (respectively, a pure epimorphism $X\longrightarrow Y$).
\end{definition}

The following theorem holds more generally. However, for simplicity, we restrict its scope to the setting in Definition \ref{Def: lfp}.

\begin{theorem}\label{lfp precover}
Let $\Xcal$ be a subcategory of a locally finitely presented additive category $\Acal$. 
\begin{enumerate}
\item If $\Xcal$ is closed under pure quotients, then it is closed under directed colimits if and only if it is closed under coproducts.
\item \cite[Theorem 4]{Krause} If $\Xcal$ is closed under pure quotients and coproducts (or, equivalently, under pure quotients and directed colimits), then $\Xcal$ is precovering.
\item\cite[Proposition 5]{Krause} If $\Acal$ admits products and if $\Xcal$ is closed under pure subobjects and products, then $\Xcal$ is preenveloping.
\end{enumerate}
\end{theorem}
\begin{proof}
We discuss only item (1).  Every coproduct is a directed colimit and, in a locally finitely presented category, it is easy to see that every directed colimit is a pure quotient of a coproduct.
\end{proof}

\begin{remark}\label{pure q pure s}
As observed in \cite{Krause,AR}, in locally finitely presented additive category, any subcategory closed under directed colimits and pure subobjects is also closed for pure quotients. Hence, if $\Xcal$ is a subcategory of a locally finitely presented category $\Acal$ closed under directed colimits and pure subobjects, then $\Xcal$ is precovering.
\end{remark}

In nice enough categories, we also get that a precovering class closed under directed colimits is, in fact, covering. Once again, taking into account our purposes, we simplify the setting in which the following theorem holds. Recall that a category is said to be \textbf{well-powered} if the subobjects of any given object form a set.

\begin{theorem}[{\cite[Section 7]{Enochs}}]\label{covering}
Let $\Acal$ be a cocomplete and well-powered abelian category and $\Xcal$ a subcategory. If $\Xcal$ is precovering in $\Acal$ and $\Xcal$ is closed under directed colimits, then $\Xcal$ is covering in $\Acal$.
\end{theorem}

Examples of cocomplete well-powered abelian categories abound. In particular, any Grothendieck category satisfies these properties (see, for example, \cite[Chapter~IV, Proposition~6.6]{Stenstrom}).

\begin{corollary}\label{cor covering}
Let $\Acal$ be a locally finitely presented abelian category and $\Xcal$ a subcategory. If $\Xcal$ is closed under pure quotients and coproducts, then it is covering.
\end{corollary}
\begin{proof}
By \cite[Section~2.4]{CBlfp}, any locally finitely presented abelian category is a Grothendieck category and therefore well-powered.  This is then a direct combination of Theorem \ref{lfp precover} and Theorem \ref{covering}.
\end{proof}

If $\Tcal$ is a triangulated category, there are also some sufficient conditions guaranteeing that certain precovering (respectively, preenveloping) subcategories are covering (respectively, enveloping). 

\begin{definition} A subcategory $\Ucal$ of $\Tcal$ is \textbf{suspended} (respectively, \textbf{cosuspended}) if it is closed under summands and extensions, and $\Ucal[1]$ is contained in $\Ucal$ (respectively, $\Ucal[-1]$ is contained in $\Ucal$).  Moreover, a subcategory $\Ucal$ is \textbf{triangulated} if it is both suspended and cosuspended. \end{definition}

The following theorem is a direct consequence of the arguments presented in \cite[Proposition 1.4]{Neeman} (see also Theorem \ref{KV}).

\begin{theorem}[{\cite[Proposition 1.4]{Neeman}}]\label{Neeman}
Let $\Ucal$ be a subcategory of a triangulated category $\Tcal$. 
\begin{enumerate}
\item If $\Ucal$ is suspended, then $\Ucal$ is precovering if and only if $\Ucal$ is covering, and if and only if  $\Ucal$ is coreflective. 
\item If $\Ucal$ is cosuspended, then $\Ucal$ is preenveloping if and only if $\Ucal$ is enveloping, and if and only if  $\Ucal$ is reflective.
\item If $\Ucal$ is a triangulated subcategory, then $\Ucal$ is both precovering and preenveloping if and only if $\Ucal$ is both covering and enveloping, and if and only if $\Ucal$ is bireflective.
\end{enumerate}
\end{theorem}

\subsection{Compactly generated triangulated categories}\label{Sec: comp gen}
We will consider triangulated categories that have a nice set of generators.
\begin{definition}
If $\Tcal$ is a triangulated category with coproducts, then
\begin{itemize}
\item an object $X$ in $\Tcal$ is \textbf{compact} if $\Hom_\Tcal(X,-)$ commutes with direct sums; 
\item $\Tcal$  is said to be \textbf{compactly generated} if the subcategory of compact objects, denoted by $\Tcal^c$, is skeletally small and generates $\Tcal$ (i.e.~if $\Hom_\Tcal(C,X)=0$ for all $C$ compact, then $X\cong 0$).
\end{itemize}
\end{definition}

Compactly generated triangulated categories admit a \textit{pure-exact structure}, which we discuss in the following. We consider the category $\mathsf{Mod}(\Tcal^c)$ of additive functors $(\Tcal^c)^{op}\longrightarrow \mathsf{Mod}(\mathbb{Z})$. This is an abelian category with directed colimits and, in fact, it is a locally finitely presented (even locally coherent) Grothendieck category. We denote by
$\mathbf{y}\colon \Tcal\longrightarrow \mathsf{Mod}(\Tcal^c)$ the functor sending an object $X$ in $\Tcal$ to the functor $\textbf{y}X:=\Hom_\Ccal(-,X)|_{\Tcal^c}$.

\begin{definition}
A triangle in a compactly generated triangulated category $\Tcal$
\begin{equation}\nonumber
\Delta\colon\xymatrix{L\ar[r]^f&M\ar[r]^g&N\ar[r]& L[1]}
\end{equation}
is said to be a \textbf{pure triangle} if $\mathbf{y}(\Delta)$ is a short exact sequence in $\mathsf{Mod}(\Tcal^c)$. In that case we say that $f$ is a \textbf{pure monomorphism}, $g$ is a \textbf{pure epimorphism}, $L$ is a \textbf{pure subobject} of $M$ and $N$ is a \textbf{pure quotient} of $M$. Moreover, an object $L$ is \textbf{pure-injective} if every pure triangle of the form $\Delta$ splits. An object $N$ is \textbf{pure-projective} if every pure triangle of the form $\Delta$ splits.
\end{definition}

It is well-known (see, for example, \cite[Sections 8 and 11]{Bel}) that the subcategory of pure-projective objects in a compactly generated triangulated category coincides with the additive closure of the compact objects $\mathsf{Add}(\Tcal^c)$. The following lemma provides a useful method to construct pure triangles.

\begin{lemma}[{\cite[Lemma 2.8]{Krause2}}]\label{pure mono}
Let $\Tcal$ be compactly generated. A triangle $\Delta$ in $\Tcal$ is pure if and only if there is a directed system of split triangles $(\Delta_i)_{i\in I}$ such that the short exact sequence $\mathbf{y}\Delta$ coincides with the short exact sequence $\varinjlim_{I} \mathbf{y}\Delta_i$.
\end{lemma}

This lemma shows that indeed the notions of pure triangle in $\Tcal$ and pure-exact sequence in the locally finitely presented category $\mathsf{Mod}(\Tcal^c)$ are naturally compatible. Note that, as a consequence of the lemma above, a morphism $f$ in $\Tcal$ is a pure monomorphism if and only if $\mathbf{y}f$ is a pure monomorphism in $\mathsf{Mod}(\Tcal^c)$. Dually, it also follows that a morphism $g$ in $\Tcal$ is a pure epimorphism if and only if $\mathbf{y}g$ is a pure epimorphism in $\mathsf{Mod}(\Tcal^c)$.

\subsection{Derived categories of dg categories}\label{Sec: Comp gen alg}
Among all compactly generated triangulated categories, we will be particularly interested in those that are equivalent to the derived category $\Dcal(\Acal)$ of a small dg category $\Acal$.  If $\Kbb$ is a commutative ring, then a $\Kbb$-linear category $\Acal$ is a \textbf{dg category} if it is enriched over the category $\Ccal(\Kbb)$ of chain complexes over $\Kbb$.  Since we will not directly make use of this enriched structure, we refer the reader to the following example for an intuitive idea of what this means and to sources such as \cite[Section~6]{Krause3} and \cite{Keller} for a detailed account.  

\begin{example}
Let $\Kbb$ be a commutative ring.  We will describe a dg category $\Ccal_{dg}(\Kbb)$ with the same objects as $\Ccal(\Kbb)$.  For any pair $X, Y$ of objects in $\Ccal(\Kbb)$ we must define a chain complex $\mathcal{H}om(X, Y)$.  For $n \in \Zbb$, we take \[\mathcal{H}om(X,Y)^n := \prod_{p\in\Zbb} \mathrm{Hom}_\Kbb(X^p, Y^{p+n})\] and the differential is defined to be $d^n(f^p) := d_Y \circ f^p - (-1)^nf^{p+1}\circ d_X$ where $f^p$ is in $\mathrm{Hom}_\Kbb(X^p, Y^{p+n})$.
\end{example}

The derived category $\Dcal(\Acal)$ is constructed as a localisation of an abelian category $\Ccal(\Acal)$, which is a generalisation of the category of chain complexes over a ring.  For the details of the definition of $\Ccal(\Acal)$ we refer to \cite[Section~2.1]{Keller}.  The category $\Dcal(\Acal)$ is obtained from $\Ccal(\Acal)$ by formally inverting a class $\Wcal$ of morphisms in $\Ccal(\Acal)$ called quasi-isomorphisms.  As one might expect, the morphisms in $\Wcal$ generalise the quasi-isomorphisms in the category of chain complexes over a ring; an exact definition may be found in \cite[Section~4.1]{Keller}. 

\section{Compactly generated algebraic triangulated categories}\label{sec: compact algebraic}

A famous theorem of Keller (\cite[Subsection 4.1 and Theorem 4.3]{Keller}) states that the class of derived categories of small dg categories coincides with the class of stable categories of Frobenius exact categories that are compactly generated. These categories are called compactly generated algebraic triangulated categories and, in this section, we explore their properties from these two distinct points of view. 

The process of formally inverting the quasi-isomorphisms in $\Ccal(\Acal)$ for a small dg category $\Acal$ may be done within the framework of model categories and we present this point of view in Subsection \ref{Sec: Model}. The category $\Dcal(\Acal)$ may also be realised as the stable category of a Frobenius exact category and there is a close relationship between the torsion pairs in $\Dcal(\Acal)$ and the cotorsion pairs in the Frobenius exact category.  More details of this construction are given in Subsection \ref{Sec: Frob ex}.   It is the interplay between these two structures underlying the category $\Dcal(\Acal)$ that will allow us to develop approximation theory in the later sections.

\subsection{Compactly generated algebraic triangulated categories as homotopy categories}\label{Sec: Model}
One way of seeing that the localisation of $\Ccal(\Acal)$ at $\Wcal$ exists is to observe that there is a model structure on $\Ccal(\Acal)$ such that the quasi-isomorphisms coincide with the weak equivalences. In particular, we will consider the so-called \textbf{projective model structure} on $\Ccal(\Acal)$ described in the following proposition.
We refer the reader to \cite{Hir,Hovey} for more details on the theory of model categories.  Given a model category $\Mcal$ we will denote its homotopy category by $\textrm{Ho}(\Mcal)$.  For a proof of the following well-known proposition, see \cite[Proposition 1.3.5]{Becker} and references therein.

\begin{proposition}\label{model in DG}
Let $\Acal$ be a small dg category.  There exists a model structure on $\Ccal(\Acal)$ such that the weak equivalences are the quasi-isomorphisms and every object is fibrant and it is uniquely determined by these properties. The homotopy category of this model structure is equivalent to $\Dcal(\Acal)$.
\end{proposition}

\begin{notation}
We will fix the following notation with respect to the projective model structure.  \begin{itemize}
\item We denote the associated localisation functor by $\pi \colon \Ccal(\Acal) \rightarrow \Dcal(\Acal)$.
\item We will denote the class of weak equivalences (i.e.~quasi-isomorphisms) by $\Wcal$.
\item We will denote the class of cofibrant objects by $\mathscr{C}$ (in \cite{Keller} Keller refers to the objects in $\mathscr{C}$ as \textit{objects with property (P)}).
\end{itemize}
\end{notation}

By \cite[Lemma~2.12]{StPo}, the category $\Ccal(\Acal)$ is a locally finitely presented Grothendieck category with enough projective objects.  In Section \ref{Sec: Approx} we will prove an analogue of Theorem \ref{lfp precover} for $\Dcal(\Acal)$. The strategy will be to determine which subcategories $\Xcal$ of $\Dcal(\Acal)$ have a preimage $\pi^{-1}(\Xcal)$ in $\Ccal(\Acal)$ satisfying the closure conditions specified in Theorem \ref{lfp precover}. We, therefore, need to discuss the products, coproducts and directed colimits (or suitable replacements of these) in both $\Ccal(\Acal)$ and $\Dcal(\Acal)$.

It is well-known (\cite[Proposition A.2.8.2]{Lurie}) that the model structure in $\Ccal(\Acal)$ described in the proposition above induces model structures in the category of functors from a small category $S$ to $\Ccal(\Acal)$. We denote this category, usually called the category of $S$-diagrams in $\Ccal(\Acal)$, by $\Ccal(\Acal)^S$ and the corresponding localisation functor by $\pi_S\colon \Ccal(\Acal)^S\longrightarrow \textrm{Ho}(\Ccal(\Acal)^S)$.

\begin{remark} It is easy to see that if $S$ is a discrete category (i.e.~if the only morphisms of $S$ are identities on objects), then $\Ccal(\Acal)^S$ is nothing but the product, indexed by $S$, of copies of the category $\Ccal(\Acal)$. In that case it then follows that the homotopy category for the induced model structure is nothing but the product of the respective homotopy categories, i.e.~$\textrm{Ho}(\Ccal(\Acal)^S)\simeq\textrm{Ho}(\Ccal(\Acal))^S\simeq \Dcal(\Acal)^S$. However, this equivalence does not generally hold for small categories which are not discrete.
\end{remark}

We recall that if $F \colon \Mcal \rightarrow \Ncal$ is a functor between model categories, then the total left derived functor (respectively, the total right derived functor) is a functor $\Lbb F \colon \mathrm{Ho}(\Mcal) \rightarrow \mathrm{Ho}(\Ncal)$ (respectively, $\Rbb F\colon \mathrm{Ho}(\Mcal) \rightarrow \mathrm{Ho}(\Ncal)$) - see \cite[Section~8.4]{Hir} for the definition of total derived functors. We  consider the following total derived functors:

\begin{itemize}
\item For each directed small category $S$ (i.e.~the set of objects of $S$ is preordered by its morphisms and every finite subset has an upper bound), we consider the total left derived functor of the directed colimit functor $\varinjlim_S \colon \Ccal(\Acal)^S \rightarrow \Ccal(\Acal)$, and we denote it by $\mathbb{L}\varinjlim_S\colon \mathrm{Ho}(\Ccal(\Acal)^S) \rightarrow \mathrm{Ho}(\Ccal(\Acal))\simeq\Dcal(\Acal) $. This derived functor turns out to be left adjoint to the natural diagonal (or constant) functor $\delta_S\colon \textrm{Ho}(\Ccal(\Acal))\simeq \Dcal(\Acal)\longrightarrow \textrm{Ho}(\Ccal(\Acal)^S)$ and we will, therefore, refer to $\Lbb\varinjlim_S$ as the \textbf{directed homotopy colimit functor}.
\item For each discrete category $S$, we consider the total right derived functor of the product functor $\prod \colon \Ccal(\Acal)^S \rightarrow \Ccal(\Acal)$, which yields a functor $\Rbb\prod \colon \Dcal(\Acal)^S \rightarrow \Dcal(\Acal)$. This derived functor is right adjoint to the natural diagonal functor $\delta_S\colon \textrm{Ho}(\Ccal(\Acal))\simeq \Dcal(\Acal)\longrightarrow \textrm{Ho}(\Ccal(\Acal)^S)\cong \Dcal(\Acal)^S$ and, therefore, it coincides with the product functor in $\Dcal(\Acal)$ (\cite[Example~1.3.11]{Hovey}).
\item For each discrete category $S$, we consider the total left derived functor of the coproduct functor $\coprod \colon \Ccal(\Acal)^S \rightarrow \Ccal(\Acal)$, which yields a functor $\Lbb\coprod \colon \Dcal(\Acal)^S \rightarrow \Dcal(\Acal)$.  This derived functor is left adjoint to the natural diagonal functor $\delta_S\colon \textrm{Ho}(\Ccal(\Acal))\simeq \Dcal(\Acal)\longrightarrow \textrm{Ho}(\Ccal(\Acal)^S)\cong \Dcal(\Acal)^S$ and, therefore, it coincides with the coproduct functor in $\Dcal(\Acal)$ (\cite[Example~1.3.11]{Hovey}).
\end{itemize}

In the same vein as \cite[Lemma~7.1(4)]{SSV} we observe that the exactness of (co)products and directed colimits in $\Ccal(\Acal)$ allow for an easy computation of (co)products and directed homotopy  colimits in $\Dcal(\Acal)$.

\begin{proposition}\label{prop: derived functors} Let $\Acal$ be a small dg category and let $S$ be a small category.
\begin{enumerate}
\item If $S$ is directed and $X$ is an object of $\Ccal(\Acal)^S$, then we have $\mathbb{L}\varinjlim_S \pi_SX \cong \pi\left(\varinjlim_S X\right)$.
\item If $S$ is discrete and $X = (X_s)_{s\in S}$ is an object of $\Ccal(\Acal)^S$, then we have $\prod_{s\in S}\pi X_s \cong \pi\left( \prod_{s\in S} X_s \right)$.
\item If $S$ is discrete and $X = (X_s)_{s\in S}$ is an object of $\Ccal(\Acal)^S$, then we have $\coprod_{s\in S} \pi X_s \cong \pi\left( \coprod_{s\in S} X_s \right)$.
\end{enumerate}
\end{proposition}
\begin{proof}
This result is well-known, but we sketch an argument along the lines of \cite[Proposition 2.2(2)]{TV}. We focus on statement (1), since the proof of statements (2) and (3) are analogous. Let $S$ be a directed small category. First we observe that, since $\varinjlim_S$ is exact in $\mathsf{Mod}(\Kbb)$, it commutes with the cohomology functor in $\Ccal(\Acal)$. This implies that weak equivalences (i.e.~quasi-isomorphisms) are preserved under directed colimits.  The statement then follows from \cite[Corollary 7.2.5]{MGroth}.
\end{proof}

\begin{corollary}\label{Cor: yoneda and hocolims}
Let $\Acal$ be a small dg category. Denote $\Tcal := \Dcal(\Acal)$ and let $\mathbf{y} \colon \Tcal \to \mathsf{Mod}(\Tcal^c)$ be the functor defined in Section \ref{Sec: comp gen}.  Then the composition $\mathbf{y}\circ \pi \colon\Ccal(\Acal)\longrightarrow \mathsf{Mod}(\Tcal^c)$ preserves directed colimits.
\end{corollary}
\begin{proof}
By \cite[Proposition 5.4]{SSV}, the compact objects are precisely the objects $X$ for which the representable functors $\Hom_\Tcal(X,-)$ send directed homotopy colimits to directed colimits in $\mathsf{Mod}(\mathbb{Z})$. Hence, $\mathbf{y}$ sends directed homotopy colimits in $\Tcal$ to directed colimits in $\mathsf{Mod}(\Tcal^c)$. Since, by Proposition \ref{prop: derived functors}, the functor $\pi$ sends directed colimits in $\Ccal(\Acal)$ to directed homotopy colimits in $\Tcal$, the corollary follows.
\end{proof}

The upshot of Proposition \ref{prop: derived functors} is that we can transfer closure conditions between $\Ccal(\Acal)$ and $\Dcal(\Acal)$.  In order to make this precise, we will need the following definitions.

\begin{definition}\label{Def: closed hocolim} Let $\Acal$ be a small dg category and $\Xcal$ a subcategory of $\Dcal(\Acal)$. \begin{enumerate}
\item The \textbf{preimage} of $\Xcal$ in $\Ccal(\Acal)$ is the subcategory \[\pi^{-1}(\Xcal) := \{ X \in \Ccal(\Acal) \mid \pi(X) \in \Xcal\}.\]
\item The subcategory $\Xcal$ is \textbf{closed under directed homotopy colimits} if for every directed small category $S$ and every object $X = (X_s)_{s\in S}$ in $\Ccal(\Acal)^S$ such that $\pi(X_s)$ is in $\Xcal$ for every $s$ in $S$, we have that $\mathbb{L}\varinjlim_S\pi_S X$ is in $\Xcal$. 
\end{enumerate}
\end{definition}

\begin{corollary}\label{colim closed}
Let $\Acal$ be a small dg category and let $\Xcal$ be a subcategory of $\Dcal(\Acal)$. The following statements hold.
\begin{enumerate}
\item $\Xcal$ is closed under products if and only if $\pi^{-1}(\Xcal)$ is closed under products.
\item $\Xcal$ is closed under coproducts if and only if $\pi^{-1}(\Xcal)$ is closed under coproducts.
\item $\Xcal$ is closed under directed homotopy colimits if and only if $\pi^{-1}(\Xcal)$ is closed under directed colimits.
\end{enumerate}
\end{corollary}

The following lemma indicates that the property of being closed under pure subobjects or pure quotients can also be transferred from $\mathcal{D}(\Acal)$ to $\Ccal(\Acal)$.

\begin{lemma}\label{transfer purity}
Let $\Acal$ be a small dg category and let $\Xcal$ be a subcategory of $\Dcal(\Acal)$. If $\Xcal$ is closed under pure quotients (respectively, pure subobjects) in $\Dcal(\Acal)$, then so is its preimage $\pi^{-1}(\Xcal)$. 
\end{lemma}
\begin{proof}
We prove the statement for pure quotients, since the other statement can be proved analogously. Let $M$ be an object in $\pi^{-1}(\Xcal)$ and suppose that there is a pure epimorphism $f\colon M\longrightarrow N$ in the locally finitely presented category $\Ccal(\Acal)$. Then we have that $f=\varinjlim_{I}f_i$ for some directed system of split epimorphisms $(f_i)_{i\in I}$ in $\Ccal(\Acal)$. Since, by Corollary \ref{Cor: yoneda and hocolims} the composition $\mathbf{y}\circ \pi$ commutes with directed colimits, it follows that $\mathbf{y}\pi(f)=\varinjlim_I \mathbf{y}\pi(f_i)$. Since $\pi(f_i)$ are split epimorphisms in $\Dcal(\Acal)$, from Lemma \ref{pure mono} we have that $\pi(f)$ is a pure epimorphism in $\Tcal$ and, thus, $\pi(N)$ must also lie in $\Xcal$. Hence $N$ lies in $\pi^{-1}(\Xcal)$ as wanted. 
\end{proof}

Now we can prove the analogous statement to Theorem \ref{lfp precover}(1).

\begin{proposition}\label{coprod colim}
Let $\Acal$ be a small dg category and let $\Xcal$ be a subcategory of $\Dcal(\Acal)$ that is closed under pure quotients. Then the subcategory $\Xcal$ is closed under directed homotopy colimits if and only if it is closed under coproducts.
\end{proposition}
\begin{proof}
Since every coproduct is a directed homotopy colimit, it remains to show that if $\Xcal$ that is closed under coproducts then it is also closed under directed homotopy colimits. By Lemma \ref{transfer purity}, the subcategory $\pi^{-1}(\Xcal)$  of $\Ccal(\Acal)$ is closed under pure quotients and it is also clearly closed under coproducts. By Theorem \ref{lfp precover}(1), it follows that $\pi^{-1}(\Xcal)$ is closed under directed colimits and, therefore, by Corollary \ref{colim closed}, we conclude that $\Xcal$ is closed under directed homotopy colimits.
\end{proof}

\subsection{Compactly generated algebraic triangulated categories as stable categories}\label{Sec: Frob ex}

 The stable category $\underline{\Fcal}$ of a Frobenius exact category $\Fcal$ is always triangulated \cite[Theorem~2.6]{Happel} and many, if not all, triangulated categories arising in algebra are of this form.  This therefore motivates the next definition.

\begin{definition}\label{Def: alg}
A triangulated category is called \textbf{algebraic} if it is equivalent to the stable category of a Frobenius exact category.
\end{definition}

We will make use of the following Frobenius exact category yielding $\Dcal(\Acal)$ as its stable category.

\begin{proposition}[{see, for example, \cite[{Theorem VIII.4.2}]{BelRei}}]\label{cofibrant Frobenius}
Let $\Acal$ be a small dg category and let $\mathscr{C}$ be the subcategory of $\Ccal(\Acal)$ consisting of cofibrant objects in the projective model structure.  The following statements hold. 
\begin{enumerate}
\item The subcategory $\mathscr{C}$ with the exact structure inherited from the abelian structure on $\Ccal(\Acal)$ is a Frobenius exact category.
\item The restriction $\pi|_{\mathscr{C}}\colon\mathscr{C} \longrightarrow \Dcal(\Acal)$ induces an equivalence of categories between $\underline{\mathscr{C}}$ and $\Dcal(\Acal)$.
\end{enumerate}
\end{proposition}

We finish this section with a brief account of the connection between the torsion pairs in $\underline{\mathscr{C}}$ and the complete cotorsion pairs in $\mathscr{C}$.  This will prove useful in Section \ref{Sec: tors}.

\begin{definition}\label{Def: torsion pair}
A pair of subcategories $(\Ucal,\Vcal)$ in a triangulated category $\Tcal$ is called a \textbf{torsion pair} if 
\begin{enumerate}
\item $\Ucal$ and $\Vcal$ are closed under summands;
\item $\Hom_\Tcal(\Ucal,\Vcal)=0$;
\item For any object $X$ in $\Tcal$ there are objects $U$ in $\Ucal$, $V$ in $\Vcal$ and a triangle 
$U\longrightarrow X\longrightarrow V\longrightarrow U[1].$
\end{enumerate}
If $(\Ucal,\Vcal)$ is a torsion pair in $\Tcal$, then we say that $\Ucal$ is a \textbf{torsion class} and $\Vcal$ is a \textbf{torsionfree class}. \end{definition}

It follows easily from the definition that $\Vcal=\Ucal^\perp$ and $\Ucal={}^\perp\Vcal$. If $\Tcal$ is an algebraic triangulated category, then torsion pairs in $\Tcal$ are in bijection with certain pairs of subcategories of $\Fcal$. Let us recall the definition of a complete cotorsion pair in an exact category.

\begin{definition}
Let $\Fcal$ be an exact category. We say that a pair of subcategories  $(\mathcal{M},\mathcal{N})$ of $\Fcal$  is a \textbf{complete cotorsion pair} if the following statements hold.
\begin{itemize}
\item $\mathcal{N}=\Ker \Ext^1_\Fcal(\mathcal{M},-)$ and $\mathcal{M}=\Ker \Ext^1_\Fcal(-,\mathcal{N})$.
\item $\mathcal{M}$ is \textbf{special precovering}, i.e.~for every object $X$ in $\Fcal$ there are objects $M_1$ in $\mathcal{M}$, $N_1$ in $\mathcal{N}$ and a conflation
$$\xymatrix{0\ar[r] & N_1\ar[r]& M_1\ar[r]^f & X\ar[r] &0.}$$ 
\item $\mathcal{N}$ is \textbf{special preenveloping}, i.e.~for every object $X$ in $\Fcal$ there are objects $N_2$ in $\mathcal{N}$, $M_2$ in $\mathcal{M}$ and a conflation
$$\xymatrix{0\ar[r] & X\ar[r]^g& N_2\ar[r] & M_2\ar[r] &0.}$$ 
\end{itemize}
\end{definition}
From the above conflations, it is easy to see that, indeed, $f$ is a $\mathcal{M}$-precover and $g$ is a $\mathcal{N}$-preenvelope. 

\begin{example}\label{example Wakamatsu}
The well-known Wakamatsu's Lemma states that every extension-closed covering subcategory of an abelian category is special precovering. In particular, if $\Acal$ is a well-powered locally finitely presented abelian category and $\Xcal$ is a subcategory closed under extensions, directed colimits and pure quotients, then by Corollary \ref{cor covering} and Wakamatsu's Lemma, $\Xcal$ is special precovering in $\Acal$.
\end{example}

\begin{example}\label{model cotorsion}
Abelian model structures are intimately related with cotorsion pairs (see \cite{Hovey2}). In particular, the projective model structure in $\Ccal(\Acal)$ from Proposition \ref{model in DG} guarantees the existence of a cotorsion pair $(\mathscr{C},\Wcal_0)$, where $\mathscr{C}$ is the subcategory of cofibrant objects and $\Wcal_0$ is the category of trivial objects (i.e.~those which are quasi-isomorphic to zero or, in other words, acyclic).
\end{example}

\begin{theorem}[{\cite[Proposition 3.3]{SaSt}}]\label{cotorsion torsion}
Let $\Tcal$ be the stable category of a Frobenius exact category $\Fcal$ and let $\phi\colon \Fcal\longrightarrow \underline{\Fcal}=\Tcal$ denote the canonical functor. There is a bijection between complete cotorsion pairs in $\Fcal$ and torsion pairs in $\Tcal$ given by the assignment $(\mathcal{M},\mathcal{N})\mapsto (\phi\mathcal{M},(\phi\mathcal{N})[1])$.
\end{theorem}

\section{Approximation theory in algebraic triangulated categories}\label{Sec: Approx}
In this section we discuss sufficient closure conditions for a subcategory of a compactly generated algebraic triangulated category to be precovering or preenveloping. The trick is to lift those closure conditions to a locally finitely presented category and use the results on approximation theory available there. All the results in this section will be stated for a compactly generated algebraic triangulated category $\Tcal$ but we wish to make reference to some notion of homotopy colimit in $\Tcal$.  We therefore make the following definition.

\begin{definition}
Let $\Tcal$ be a compactly generated algebraic triangulated category.  We will say that a subcategory $\Xcal$ is \textbf{closed under homotopy colimits} if there exists a small dg category $\Acal$ and an equivalence $F \colon \Tcal \rightarrow \Dcal(\Acal)$ such that $F(\Xcal)$ is closed under homotopy colimits in $\Dcal(\Acal)$.
\end{definition}

 The following theorem is a triangulated analogue of Theorem \ref{lfp precover}.  

\begin{theorem}\label{approx for def}
Let $\Xcal$ be a subcategory of a compactly generated algebraic triangulated category $\Tcal$.
\begin{enumerate}
\item If $\Xcal$ is closed under pure quotients and coproducts, then $\Xcal$ is precovering.
\item If $\Xcal$ is closed under pure subobjects and products, then $\Xcal$ is preenveloping.
\end{enumerate}
\end{theorem}
\begin{proof}
(1) Since there exists a small dg category $\Acal$ such that $\Tcal \simeq \Dcal(\Acal)$, it suffices to prove (1) for $\Dcal(\Acal)$.  So let $\Xcal$ be a subcategory of $\Dcal(\Acal)$ that is closed under pure quotients and coproducts.  Let $\pi^{-1}(\Xcal)$ denote the preimage of $\Xcal$ in $\Ccal(\Acal)$ (as in Definition \ref{Def: closed hocolim}). Then, by Corollary \ref{colim closed} and Lemma \ref{transfer purity}, the subcategory $\pi^{-1}(\Xcal)$ is closed under coproducts and pure quotients. Thus, it follows from Theorem \ref{lfp precover}(2) that $\pi^{-1}(\Xcal)$ is precovering in $\Ccal(\Acal)$. We use this fact to show that $\Xcal$ is precovering in $\Dcal(\Acal)$. 

Let $\pi(M)$ be an arbitrary object in $\Dcal(\Acal)$ and consider a $\pi^{-1}(\Xcal)$-precover $p\colon X\rightarrow M$ of $M$ in $\Ccal(\Acal)$. We show that $\pi(p) \colon \pi(D) \rightarrow \pi(M)$ is a $\Xcal$-precover.  Let $f \colon \pi(E) \rightarrow \pi(M)$ be any morphism from an object in $\Xcal$ to $\pi(M)$ and let $c \colon C_E \rightarrow E$ be a cofibrant replacement of $E$.  Note that $\pi(c)$ is an isomorphism and so $f$ factors through $\pi(p)$ if and only if $f\circ \pi(c)$ factors through $\pi(p)$.  Moreover, the morphism $f \circ \pi(c)$ is of the form $\pi(g)$ for some $g\colon C_E \rightarrow M$ because $C_E$ is cofibrant.  Finally, we observe that $C_E$ is contained in $\pi^{-1}(\Xcal)$ and so $g$ factors through $p$.  Thus we have that $\pi(g) = f \circ \pi(c)$ (and therefore $f$) factors through $\pi(p)$ as required.

(2) The proof is very similar to (1), the only significant difference being that to build a $\Xcal$-preenvelope of an object $\pi(M)$ in $\Dcal(\Acal)$ we should consider a $\pi^{-1}(\Xcal)$-preenvelope of a cofibrant replacement $C_M$ of $M$. Indeed, if $p\colon C_M\longrightarrow M$ is such a cofibrant replacement and $f\colon C_M\longrightarrow X$ is a $\pi^{-1}(\Xcal)$-preenvelope of $C_M$, then the composition $\pi(f)\circ \pi(p)^{-1}$ can be shown as above to be a $\Xcal$-preenvelope of $\pi(M)$.
\end{proof}

\begin{corollary}
If $\Xcal$ is closed under pure quotients and directed homotopy colimits, then $\Xcal$ is precovering.
\end{corollary}
\begin{proof}
This is immediate from Theorem \ref{approx for def} and Proposition \ref{coprod colim}.
\end{proof}

\begin{remark}\label{pure q pure s2}
Using Remark \ref{pure q pure s}, it can be seen that a subcategory $\Xcal$ of a compactly generated algebraic triangulated category $\Tcal$ that is closed under directed homotopy colimits and pure subobjects is also closed under pure quotients. Hence, such a subcategory $\Xcal$ is also precovering.
\end{remark}

\begin{example}\label{Ex PGen}
Consider an object $X$ in a compactly generated algebraic triangulated category $\Tcal$. Let $\mathsf{PGen}(X)$ denote the subcategory of $\Tcal$ whose objects are pure quotients of coproducts of $X$. Dually, let $\mathsf{PCogen}(X)$ denote the subcategory of $\Tcal$ whose objects are pure subobjects of products of $X$. Using the fact that products and coproducts preserve pure epimorphisms and pure monomorphisms, it then follows from Theorem \ref{approx for def} that  $\mathsf{PGen}(X)$ is precovering and that $\mathsf{PCogen}(X)$ is preenveloping.
\end{example}

Note that, in view of Remark \ref{pure q pure s2}, any subcategory closed under pure subobjects, directed homotopy colimits and products will be simultaneously precovering and preenveloping. It turns out that these closure conditions actually characterise an important type of subcategories.

\begin{definition}
Let $\Tcal$ be a compactly generated triangulated category. An additive (covariant) functor $F\colon\Tcal\longrightarrow \mathsf{Mod}(\mathbb{Z})$ is said to be \textbf{coherent} if there is an exact sequence
$$\xymatrix{\Hom_\Tcal(Y,-)\ar[rr]^{\Hom_\Tcal(f,-)}&&\Hom_\Tcal(X,-)\ar[rr]&&F\ar[rr]&&0}$$
with $f\colon X\longrightarrow Y$ a map in $\Tcal^c$. A subcategory $\Xcal$ of $\Tcal$ is said to be \textbf{definable} if there is a set of coherent functors $(F_i)_{i\in I}$ such that $F_i(X)=0$ for all $i$ in $I$ if and only if $X$ lies in $\Xcal$.
\end{definition}

The following theorem was originally stated in the more general setting of stable derivators (\cite[Theorem 3.11]{Laking}). Derived categories of small dg categories lie at the base of a stable derivator since they are homotopy categories of nice model categories (\cite[Proposition 1.30]{Groth}). We may, therefore, restate the theorem in our setting and furthermore, we use Proposition \ref{coprod colim} to `simplify' the original statement.
\begin{theorem}\label{closure for definable}
The following are equivalent for a subcategory $\Xcal$ of a compactly generated algebraic triangulated category.
\begin{enumerate}
\item $\Xcal$ is definable;
\item $\Xcal$ is closed under pure subobjects, products and directed homotopy colimits.
\item $\Xcal$ is closed under pure subobjects, products and pure quotients.
\end{enumerate}
\end{theorem}
\begin{proof}
(1) $\Leftrightarrow$ (2): This statement holds in our setting by \cite[Theorem 3.11]{Laking}. 

 (2) $\Rightarrow$ (3):  It follows from Remark \ref{pure q pure s2} that if $\Xcal$ is closed under pure subobjects and directed homotopy colimits, then it is also closed under pure quotients.

(3) $\Rightarrow$ (2): Since any coproduct is a pure subobject of the corresponding product, if $\Xcal$ is closed under pure subobjects and products it is then closed under coproducts as well. If additionally $\Xcal$ is closed under pure quotients, then by Proposition \ref{coprod colim} it is also closed under directed homotopy colimits.
\end{proof}

As expected, definable subcategories therefore have nice approximation-theoretic properties.

\begin{corollary}\label{definable precover preenvelope}
If $\Xcal$ is a definable subcategory of a compactly generated algebraic triangulated category, then $\Xcal$ is precovering and preenveloping.
\end{corollary}
\begin{proof}
This is an immediate consequence of Theorem \ref{closure for definable} and Theorem \ref{approx for def}.
\end{proof}

\begin{remark}
Note that it was also shown in \cite[Proposition 4.5]{AMV3} that definable subcategories are preenveloping in any compactly generated triangulated category. 
\end{remark}

Using Theorem \ref{Neeman}, we provide sufficient conditions for a subcategory to be covering or enveloping.

\begin{corollary}\label{approx t-str}
Let $\Xcal$ be a subcategory of a compactly generated algebraic triangulated category.
\begin{enumerate}
\item If $\Xcal$ is suspended and closed under coproducts and pure quotients, then $\Xcal$ is covering.
\item If $\Xcal$ is cosuspended and closed under products and pure subobjects, then $\Xcal$ is enveloping.
\end{enumerate}
In particular, if $\Xcal$ is triangulated and definable, then it is both covering and enveloping.
\end{corollary}
\begin{proof}
Statements (1) and (2) follow from Theorem \ref{Neeman} and Theorem \ref{approx for def}. The last statement follows from Theorem \ref{closure for definable}, combining (1) and (2) (or, alternatively, from Theorem \ref{Neeman} and Corollary \ref{definable precover preenvelope}).
\end{proof}

\begin{remark}\label{rem pure sub}
Suspended subcategories closed under coproducts are also closed under directed homotopy colimits. This is a conjugation of \cite[Proposition 2.4]{StPo} and \cite[Theorem 7.13]{PS}. It then follows Remark \ref{pure q pure s2} and part (1) of the corollary above that if $\Xcal$ is suspended and closed under coproducts and pure subobjects, then $\Xcal$ is covering. 
\end{remark}

Also by Theorem \ref{Neeman}, subcategories satisfying the conditions of the corollary above are part of certain torsion pairs (namely t-structures). That is the topic of our next section.

\section{Torsion pairs in algebraic triangulated categories}\label{Sec: tors}
Typical examples of (pre)covering and (pre)enveloping subcategories of triangulated categories occur in torsion pairs (see Definition \ref{Def: torsion pair}).  Clearly, every torsion class is precovering and every torsionfree class is preenveloping. In the next theorem we will make use of the results of Section \ref{Sec: Approx}, as well as the correspondence given by Theorem \ref{cotorsion torsion} in order to establish sufficient conditions, via closure properties, for a subcategory to be a torsion class. We will need the following lemma.

\begin{lemma}\label{cc pair}
Let $\Fcal$ be an exact category with enough injectives and $\mathcal{M}$ a subcategory of $\Fcal$. Then there is a complete cotorsion pair $(\mathcal{M}, \mathcal{N})$ if and only if $\mathcal{M}$ is a special precovering and idempotent complete subcategory of $\Fcal$.
\end{lemma}
\begin{proof}
It is clear from the definition that a subcategory $\mathcal{M}$ which is part of a complete cotorsion pair $(\mathcal{M},\mathcal{N})$ must be special precovering and idempotent complete. We prove the converse statement. Indeed, let $\mathcal{M}$ be special precovering and idempotent complete, and consider the class $\mathcal{N} := \Ker \Ext^1_\Fcal(\mathcal{M},-)$. We must show that $\mathcal{M}=\Ker \Ext^1_\Fcal(-,\mathcal{N})$.  Clearly $\mathcal{M}$ is contained in $\Ker \Ext^1_\Fcal(-,\mathcal{N})$. Suppose $X$ is in $\Ker \Ext^1_\Fcal(-,\mathcal{N})$ and let 
$$\xymatrix{0\ar[r] & N_1\ar[r]& M_1\ar[r]^f & X\ar[r] &0}$$ be a conflation of $\Fcal$, where $f$ is a special precover of $X$.  Then, since $\Ext^1_\Fcal(X, N_1) = 0$, the sequence is split and so $X$ lies in $\mathcal{M}$.

Finally we must obtain a special $\mathcal{N}$-preenvelope for each object $Y$ in $\Fcal$.  Since $\Fcal$ has enough injectives, we may apply Salce's Lemma (see, for example, \cite[Lemma 2.2.6]{GoebelTrlifaj}) to obtain a conflation $$\xymatrix{0\ar[r] & Y\ar[r]^g& N_2\ar[r] & M_2 \ar[r] &0}$$ with $M_2$ in $\mathcal{M}$ and $N_2$ in $\mathcal{N}$.
\end{proof}

We can now prove the main theorem of this section.  

\begin{theorem}\label{closure tors class}
Let $\Tcal$ be a compactly generated algebraic triangulated category. Then every subcategory $\Xcal$ that is closed under extensions, pure quotients and coproducts is a torsion class. In particular, every definable extension-closed subcategory is a torsion class.
\end{theorem}
\begin{proof}
Let $\Acal$ be a small dg category such that $\Tcal=\Dcal(\Acal)$ for which the subcategory $\Xcal$ of $\Tcal$ satisfies the closure conditions above. Let $\mathscr{C}$ be the Frobenius exact subcategory of $\Ccal(\Acal)$ as in Proposition \ref{cofibrant Frobenius}. By Theorem \ref{cotorsion torsion}, it is sufficient to show that there is a cotorsion pair $(\mathcal{M},\mathcal{N})$ in $\mathscr{C}$ such that $\pi(\mathcal{M}) = \Xcal$.  So let $\pi^{-1}(\Xcal)$ be the preimage of $\Xcal$ (see Definition \ref{Def: closed hocolim}) and let $\mathcal{M} = \mathscr{C}\cap\pi^{-1}(\Xcal)$.  By \cite[Proposition~1.1.15]{Becker}, we have that $\pi(\mathcal{M}) = \Xcal$. 
 
We will show that $\mathcal{M}$ is idempotent complete and special precovering, and the result then follows from Lemma \ref{cc pair}. First, since $\Xcal$ is closed under pure quotients, then so is $\pi^{-1}(\Xcal)$ by Lemma \ref{transfer purity}. In particular, $\pi^{-1}(\Xcal)$ is closed under direct summands and, since $\mathscr{C}$ also has this property, it follows that indeed $\mathcal{M}$ is idempotent complete. So it only remains to show that $\mathcal{M}$ is special precovering. Note that, since $\pi^{-1}(\Xcal)$ is also closed under coproducts, it follows from Corollary \ref{cor covering} that it is covering in $\Ccal(\Acal)$ and, moreover, it contains all the projective objects of $\Ccal(\Acal)$ (i.e.~the trivial cofibrant objects in $\Ccal(\Acal)$). Since $\pi^{-1}(\Xcal)$ is also extension-closed it then follows from Wakamatsu's Lemma (see Example \ref{example Wakamatsu}) that for each $N$ in $\Ccal(\Acal)$ there is a short exact sequence
$$\xymatrix{0\ar[r]&K\ar[r]&X\ar[r]^f&N\ar[r]&0}$$
with $K$ in $\mathsf{Ker}\Ext^1_{\Ccal(\Acal)}(\pi^{-1}(\Xcal),-)$ and $X$ in $\pi^{-1}(\Xcal)$. Moreover, the subcategory of cofibrant objects $\mathscr{C}$ is part of a complete cotorsion pair $(\mathscr{C},\Wcal_0)$, where $\Wcal_0$ is the subcategory of $\Ccal(\Acal)$ formed by the acyclic dg modules (see Example \ref{model cotorsion}).  Therefore there exists a short exact sequence
$$\xymatrix{0\ar[r]&W\ar[r]&C\ar[r]^g&X\ar[r]&0}$$
with $W$ in $\Wcal_0$ and $C$ in $\mathscr{C}$.
We claim that $f\circ g\colon C\longrightarrow N$ is a special $\mathcal{M}$-precover. First observe that indeed $C$ lies in $\mathcal{M}$ since $\pi(C)\cong \pi(X)$. It now suffices to check that $\Ext^1_{\Ccal(\Acal)}(\mathcal{M},\Ker(f\circ g))=0$. The Snake Lemma guarantees that there is a short exact sequence
$$\xymatrix{0\ar[r]&W\ar[r]&\Ker(f\circ g)\ar[r]&K\ar[r]&0.}$$
Applying $\Ext^1_{\Ccal(\Acal)}(M,-)$ to the sequence, with $M$ in $\mathcal{M}$, yields $\Ext^1_{\Ccal(\Acal)}(M, \Ker(f\circ g))=0$ as wanted.
\end{proof}

\begin{remark}
It is well-known that in a well-powered cocomplete abelian category, torsion classes are precisely those which are closed under extensions, coproducts and quotients. The theorem above provides us with a triangulated analogue for the sufficiency of these closure conditions. Note, however, that while torsion classes in abelian categories are covering, this is not necessarily the case in triangulated categories.
\end{remark}

This theorem allows us to construct a series of torsion pairs. 

\begin{corollary}\label{ex torsion pairs}
Let $\Tcal$ be a compactly generated algebraic triangulated categories. Then the following classes are torsion classes in $\Tcal$.
\begin{enumerate}
\item $P^{\perp}$ for any pure-projective object $P$;
\item ${}^{\perp}E$ for any pure-injective object $E$;
\item $\mathsf{PGen}(P)$ for a pure-projective object $P$ satisfying $\Hom_\Tcal(P,P^{(I)}[1])=0$ for any set $I$.
\end{enumerate}
\end{corollary}
\begin{proof}
(1) If $P$ is pure-projective, then it is easy to see that $P^\perp$ is closed under pure subobjects, pure quotients and products. Therefore, by Theorem \ref{closure for definable}, $P^\perp$ is definable. In particular it is also closed under coproducts. Clearly  $P^\perp$ is extension-closed and, therefore, it satisfies the assumptions of Theorem \ref{closure tors class}.

(2) If $E$ is pure-injective, then it is easy to see that ${}^\perp E$ satisfies the assumptions of Theorem \ref{closure tors class}.

(3) By Example \ref{Ex PGen}, the subcategory $\mathsf{PGen}(P)$ is closed under coproducts and pure quotients. In view of Theorem \ref{closure tors class} it remains to show that it is extension-closed. First we observe that, since  $\Hom_\Tcal(P,P^{(I)}[1])=0$, we in fact have that  $\Hom_\Tcal(P,M[1])=0$ for any $M$ in $\mathsf{PGen}(P)$. Indeed, if $f\colon P^{(I)}\longrightarrow M$ is a pure epimorphism, then so is $f[1]$ and, therefore, $\Hom_\Tcal(P,f[1])$ is surjective (since $P$ is pure-projective). This proves the claim. Finally, consider a triangle
$$\xymatrix{X\ar[r]^r&Y\ar[r]^s&Z\ar[r]^t&X[1]}$$
with $X$ and $Z$ in $\mathsf{PGen}(P)$. Given a pure epimorphism $q\colon P^{(J)}\longrightarrow Z$, since $t\circ q=0$, there is a morphism $\tilde{q}\colon P^{(J)}\longrightarrow Y$ such that $t\circ \tilde{q}=q$. Finally, given also a pure epimorphism $p\colon P^{(I)}\longrightarrow X$, it follows that the morphism $(r\circ p,-\tilde{q})\colon P^{(I)}\oplus P^{(J)}\longrightarrow Y$ is a pure epimorphism, proving that $Y$ lies in $\mathsf{PGen}(P)$. 
\end{proof}

\begin{remark}
We observe that the torsion classes considered above relate to the following three examples.
\begin{enumerate}
\item It is shown in \cite{AI} that, in a compactly generated triangulated category $\Tcal$, the orthogonal $\Scal^\perp$ to a set $\Scal$ of compact objects (or, equivalently, the orthogonal $S^\perp$ to the pure-projective object $S$ obtained as the coproduct of all objects in $\Scal$) is a torsionfree class. In \cite{StPo} it is then shown that the subcategory $\Scal^\perp$ is also a torsion class if the category $\Tcal$ is furthermore assumed to be algebraic. The statement (1) of the Corollary above is a slight generalisation of this second fact. 
\item Statement (2) of the above Corollary is, to some extent, dual to that of \cite{AI}, where it is shown that compact objects generate torsion pairs. Here we show that pure-injective objects cogenerate torsion pairs in triangulated categories, just as they cogenerate cotorsion pairs in module categories (see, for example, \cite{GoebelTrlifaj}). Certain objects cogenerating torsion pairs were studied in \cite{OPS} - these are called 0-cocompact. All known examples of 0-cocompact objects are pure-injective. On the other hand, it follows from the definition that a 0-cocompact object $X$ has the property that ${}^\perp X$ is closed under countable products. There are examples of pure-injective objects that do not satisfy this property (see \cite{AH}) and that, therefore, are not 0-cocompact. In general, however, the relation between the notions of pure-injectivity and 0-cocompactness remains rather mysterious.
\item Statement (3) of the above Corollary is analogous to a statement well-known for module categories, namely that given a module $M$ over a ring $R$ such that $\mathsf{Ext}^1_R(M,\mathsf{Gen}(M))=0$, then $\mathsf{Gen}(M)$ is a torsion class - see for example \cite[Lemma 2.3]{AMV1}.
\end{enumerate}
\end{remark}

Theorem \ref{closure tors class} allows us to give a criterion for a torsionfree class to be also a torsion class. 
 
\begin{definition} If a subcategory $\Zcal$ of $\Tcal$ is simultaneously a torsion and a torsionfree class, then we call $\Zcal$ a \textbf{TTF class} and $({}^\perp\Zcal,\Zcal,\Zcal^\perp)$ a \textbf{TTF triple}.  Moreover, we say that the torsion pairs $({}^\perp\Zcal,\Zcal)$ and $(\Zcal,\Zcal^\perp)$ are \textbf{adjacent}.  \end{definition}

\begin{proposition}\label{Cor: def torsf is TTF}
Let $\Tcal$ be a compactly generated algebraic triangulated category. Every torsionfree class closed under coproducts and pure quotients is a TTF class. In particular, every definable torsion free class is a TTF class.
\end{proposition}
\begin{proof}
This follows from Theorem \ref{closure tors class} using the fact that torsionfree classes are extension-closed.
\end{proof}

The question of when torsion pairs admit adjacent torsion pairs is extensively approached in \cite[Section 3]{Bondarko}. TTF classes are interesting from many different points of view. Examples of suspended or cosuspended TTF classes arise naturally in silting theory (see, for example, \cite{MV2}). In this case, the torsion pairs involved are t-structures or co-t-structures. 

\begin{definition} If $(\Ucal,\Vcal)$ is a torsion pair, then we say that it is 
\begin{itemize}
\item a \textbf{t-structure} if $\Ucal$ is suspended (or, equivalently, if $\Vcal$ is cosuspended); 
\item a \textbf{co-t-structure} if $\Ucal$ is cosuspended (or, equivalently, if $\Vcal$ is suspended);
\item \textbf{left nondegenerate} if $\bigcap_{i\in \Z} \Ucal[i] = 0$; 
\item \textbf{right nondegenerate} if $\bigcap_{i\in \Z} \Vcal[i] =0$.
\end{itemize}
If $(\Ucal,\Vcal)$ is a t-structure, we say that $\Ucal$ is an \textbf{aisle}, $\Vcal$ is a \textbf{coaisle} and  $\Ucal[-1]\cap \Vcal$ is the \textbf{heart} of $(\Ucal,\Vcal)$.
\end{definition}

It is well-known (see \cite{BBD}) that the heart $\Hcal$ of a t-structure in a triangulated category $\Tcal$ is an abelian category and that there is a naturally defined cohomological functor $H^0\colon\Tcal\longrightarrow \Hcal$. Moreover, in this case, torsion precovers and torsionfree preenvelopes can be chosen functorially (and we will denote the approximation triangle for an object $X$ in $\Tcal$ by $U_X \rightarrow X \rightarrow V_X\rightarrow U_X[1]$).  Although this statement goes back to the definition of t-structure in \cite{BBD}, it can also be obtained as a consequence of Theorem \ref{Neeman}. In fact, the following is a well-known result due to Keller and Vossieck.

\begin{theorem}[{\cite[Proposition 1.1]{KV}}]\label{KV}
Let $\Tcal$ be a triangulated category and $\Ucal$ a suspended subcategory of $\Tcal$. The following statements are equivalent.
\begin{enumerate}
\item There is a t-structure $(\Ucal,\Ucal^\perp)$ is $\Tcal$;
\item The inclusion of $\Ucal$ in $\Tcal$ admits a right adjoint;
\item The inclusion of $\Ucal^\perp$ in $\Tcal$ admits a left adjoint.
\end{enumerate}
\end{theorem}

Theorem \ref{KV} together with Theorem \ref{Neeman} allows us to restate Corollary \ref{approx t-str} in the following way.

\begin{proposition}\label{adjoints}
Let $\Xcal$ be a subcategory of a compactly generated algebraic triangulated category.
\begin{enumerate}
\item If $\Xcal$ is suspended and closed under coproducts and pure quotients, then $(\Xcal,\Xcal^\perp)$ is a t-structure.
\item If $\Xcal$ is cosuspended and closed under products and pure subobjects, then $({}^\perp\Xcal,\Xcal)$ is a t-structure.
\end{enumerate}
In particular, if $\Xcal$ is triangulated and definable, then both $({}^\perp\Xcal,\Xcal)$ and $(\Xcal,\Xcal^\perp)$ are t-structures i.e.~$\Xcal$ is a triangulated TTF class.
\end{proposition} 

For a subcategory $\Xcal$ of a triangulated category $\Tcal$ and $I$ a subset of $\Z$, we write:
 \[ \Xcal^{\perp_I} := \bigcap_{i\in I} (\Xcal[-i])^\perp \hspace{20pt} {}^{\perp_I}\Xcal := \bigcap_{i\in I} {}^\perp(\Xcal[i]). \] 
For each $i$ in $\Z$, we will denote the set $\{j\in \Z \mid j > i\}$ by $>i$; similarly for $<i$, $\leq i$, $\geq i$, $i$. The following corollary uses the proposition above to build t-structures out of pure-injective objects in a minimal way.

\begin{corollary}\label{t-structure from pure-injective}
Let $\Tcal$ be a compactly generated algebraic triangulated category and $E$ a pure-injective object in $\Tcal$. Then there a t-structure $({}^{\perp_{\leq 0}}E, \Vcal_E)$ and $\Vcal_E$ is the smallest coaisle containing $E$.
\end{corollary}
\begin{proof}
By Proposition \ref{adjoints}(1), there is a t-structure of the form  $({}^{\perp_{\leq 0}}E, \Vcal_E)$ and, clearly, $E$ lies in $\Vcal_E$. Suppose now that $\Vcal$ is a coaisle containing $E$. Then $E[k]$ lies in $\Vcal$ for all $k\leq 0$ and, therefore, ${}^\perp\Vcal$ is contained in ${}^{\perp_{\leq 0}}E$. Hence, we have $\Vcal_E\subseteq\Vcal$.
\end{proof}
\section{Definability and recollements}\label{Sec: Recoll}

In this section we consider the case of triangulated TTF classes, i.e.~triangulated subcategories which are both torsion and torsionfree classes. We know from Proposition \ref{adjoints} that triangulated TTF classes are bireflective subcategories. These turn out to be related to certain diagrams of functors called recollements. Let us recall and relate all these concepts.

\begin{definition}\label{Def: recollement}
A diagram of triangulated categories and triangle functors of the form
\begin{equation}\label{rec}
\xymatrix@C=0.5cm{\Bcal \ar[rrr]^{i_*} &&& \Tcal \ar[rrr]^{j^*}  \ar @/_1.5pc/[lll]_{i^*}  \ar @/^1.5pc/[lll]_{i^!} &&& \Ycal\ar @/_1.5pc/[lll]_{j_!} \ar @/^1.5pc/[lll]_{j_*} }
\end{equation}
is said to be a \textbf{recollement} of $\Tcal$ if
 \begin{enumerate}
\item $(i^*, i_*, i^!)$ and $(j_!, j^*, j_*)$ are adjoint triples;
\item $i_*$, $j_!$ and $j_*$ are fully faithful; and
\item $\im i_* = \Ker j^*$.
\end{enumerate}
Two recollements given by triples $(i^*, i_*, i^!)$, $(j_!, j^*, j_*)$ and $(i^{\prime*}, i_*^\prime, i^{\prime!})$, $(j_!^\prime, j^{\prime*}, j_*^\prime)$ respectively  are \textbf{equivalent} if $\im i_* = \im i_*^\prime$, $\im j_! = \im j_!^\prime$ and $\im j_* = \im j_*^\prime$.
\end{definition}

Recall that given a recollement as above, for each object $X$ in the triangulated $\Tcal$, the units and counits of the adjunctions yield triangles as follows
\begin{equation}\nonumber
j_!j^*X\longrightarrow X\longrightarrow i_*i^*X\longrightarrow j_!j^*X[1]\ \ \ \ \ \ \ \ \ i_*i^!X\longrightarrow X\longrightarrow j_*j^*X\longrightarrow i_*i^!X[1].
\end{equation}

Triangulated TTF triples also turn out to be related to the notion of (co)smashing subcategory.

\begin{definition}
A triangulated subcategory $\Scal$ of a triangulated category $\Tcal$ with products and coproducts is said to be \textbf{smashing} if $\Scal$ is coreflective and $\Scal^\perp$ is closed under coproducts. Dually, we say that a triangulated subcategory $\Ccal$ is \textbf{cosmashing} if it is reflective and ${}^\perp\Ccal$ is closed under products.
\end{definition}

We now state the following equivalent descriptions of triangulated TTF classes. Most of the conditions are well-known, however we provide an argument for the sake of completion.

\begin{proposition}\label{Prop: eq bireflective}
Let $\Tcal$ be a compactly generated algebraic triangulated category.  The following are equivalent for a triangulated subcategory $\Bcal$ of $\Tcal$.
\begin{enumerate}
\item $\Bcal$ is a TTF class in $\Tcal$;
\item $\Bcal=\Scal^\perp$ for a smashing subcategory $\Scal$;
\item $\Bcal={}^\perp\Ccal$ for a cosmashing subcategory $\Ccal$;
\item $\Bcal$ is bireflective;
\item There is a recollement of $\Tcal$ of the form (\ref{rec});
\item $\Bcal$ is definable;
\item $\Bcal$ is closed under products and pure subobjects.
\end{enumerate}
\end{proposition}
\begin{proof}
(1) $\Leftrightarrow$ (2) $\Leftrightarrow$ (3): We refer to \cite[Corollary 3.15]{NS} and references therein.

(1) $\Leftrightarrow$ (4): This follows, for example, from Theorem \ref{KV}.

(4) $\Rightarrow$ (5): Define $i_*\colon\Bcal\longrightarrow \Tcal$ to be the inclusion functor, $\Ycal:=\Tcal/\Bcal$ to be the Verdier quotient and let $j^*\colon \Tcal\longrightarrow \Ycal$ be the corresponding Verdier quotient functor. Since $\Bcal$ is bireflective, it then follows from \cite[Theorem 2.1]{CPS} that $j^*$ has both left and right adjoints and that these are fully faithful.

(5) $\Rightarrow$ (4): This is clear by definition of bireflective subcategory.

(1) $\Rightarrow$ (6): Let $(\Scal, \Bcal, \Ccal)$ be a triangulated TTF triple. Consider the ideal $\mathcal{I}$ of morphisms in $\Tcal^c$ consisting of morphisms factoring through the smashing subcategory $\Scal$.  By \cite[Theorem 12.1]{Krause2.5}, the class $\mathcal{I}^\perp := \{ B \in \Tcal \mid \Hom_\Tcal(g, B) = 0 \text{ for all } g \in \mathcal{I}\}$ coincides with the class $\Bcal$.  For each $g$ in $\mathcal{I}$, we may complete to a triangle 
$$L \overset{g}{\longrightarrow} M \overset{f}{\longrightarrow} N \longrightarrow L[1]$$ 
and note that $\Hom_\Tcal(g, B) = 0$ if and only if $\Hom_\Tcal(f, B)$ is surjective.  That is, if and only if the coherent functor $F_f := \Coker(\Hom_\Tcal(f, -) )$ vanishes on $B$.  Thus, the class $\Bcal$ is definable.

(6) $\Rightarrow$ (1): This follows from Proposition \ref{adjoints}.

(6) $\Rightarrow$ (7): This follows from Theorem \ref{closure for definable}.

(7) $\Rightarrow$ (6): Since $\Bcal$ closed under pure subojects and products, it is also closed under coproducts. Since it is furthermore triangulated, it is then closed under directed homotopy colimits (see Remark \ref{rem pure sub}), and the result then follows from Theorem \ref{closure for definable}.
\end{proof}

\begin{remark}
The bijective correspondence between triangulated TTF classes, smashing subcategories and equivalence classes of recollements is well-understood for a large class of triangulated categories (see, for example, \cite{Nicolas}). Moreover, in a compactly generated triangulated category, the equivalence between (2) and (6) can be found in \cite{Krause2.5}. It boils down to the fact that smashing subcategories of compactly generated triangulated categories are determined by certain ideals of morphisms between compact objects. We use those arguments to prove (1) $\Rightarrow$ (6), and provide new arguments for the converse implication.
\end{remark}

\begin{remark}
The Telescope Conjecture for compactly generated triangulated categories asserts that every smashing subcategory is compactly generated, i.e.~that given a smashing subcategory $\Scal$ of a compactly generated triangulated category $\Tcal$, there is a set of compact objects $\Kcal$ such that $\Scal^\perp=\Kcal^\perp$. If $\Tcal$ is furthermore algebraic, Proposition \ref{Prop: eq bireflective} allows us to restate the Telescope Conjecture in an equivalent way: for every triangulated subcategory $\Bcal$ closed under products and pure subobjects, there is a set of compact objects $\Kcal$ such that $\Bcal=\Kcal^\perp$. The Telescope Conjecture is, however, known to be false for general compactly generated algebraic triangulated categories (see \cite{Keller2} and \cite{BaSt} for counterexamples). Still, it remains a difficult problem to identify which triangulated categories satisfy the Telescope Conjecture. 
\end{remark}

Next we observe that given an extension-closed definable subcategory of a compactly generated algebraic triangulated category (which by Theorem \ref{closure tors class} is a torsion class), we can build a subcategory satisfying the equivalent conditions of the proposition above.

\begin{corollary}\label{definable to bireflective}
Let $\Xcal$ be a subcategory of a compactly generated algebraic triangulated category $\Tcal$. If $\Xcal$ is closed under extensions, products and pure subobjects, then $\cap_{i\in\mathbb{Z}}\Xcal[i]$ is a triangulated TTF class.
\end{corollary}
\begin{proof}
It is clear that the intersection is closed under extensions, products and pure subobjects. Now observe that if $X$ lies in $\Xcal[i]$ for all $i$ in $\mathbb{Z}$, then $X[k]$ lies in $\Xcal[k+i]$ for all $i$ in $\mathbb{Z}$, thus showing that $\cap_{i\in\mathbb{Z}}\Xcal[i]$ is triangulated. The assertion then follows directly from Proposition \ref{Prop: eq bireflective}.
\end{proof}

Recollements of triangulated categories are used to glue torsion pairs. We recall the following well-known theorem (whose proof is the same as the original statement made for t-structures in \cite{BBD}). Given two subcategories $\Wcal$ and $\Zcal$ of $\Tcal$, we denote by $\Wcal\ast \Zcal$ the subcategory of $\Tcal$ whose objects are those $X$ in $\Tcal$ for which there are objects $W$ in $\Wcal$, $Z$ in $\Zcal$ and a triangle
$$W\longrightarrow X\longrightarrow Z\longrightarrow W[1].$$

\begin{theorem}[{\cite[Th\'eor\`eme 1.4.10, Proposition 1.4.12]{BBD}}]\label{glueing}
Let $\Tcal$ be a triangulated category and consider a recollement of $\Tcal$ of the form (\ref{rec}), and denote by $(\Scal,\Bcal,\Ccal)$ its associated triangulated TTF triple.
\begin{enumerate}
\item If $(\Wcal,\Zcal)$ and $(\Ucal,\Vcal)$ are torsion pairs in $\Bcal$ and $\Ycal$ respectively, then there is a torsion pair in $\Tcal$ of the form $(j_!\Ucal\ast i_*\Wcal, i_*\Zcal\ast j_*\Vcal)$. 
\item A torsion pair $(\Mcal,\Ncal)$ in $\Tcal$ is obtained by the construction in (1) if and only if $j_!j^*\Mcal\subseteq \Mcal$ and if and only if $j_*j^*\Ncal\subseteq \Ncal$, in which case the torsion pairs that give rise to it are $(j^*\Mcal,j^*\Ncal)$ in $\Ycal$ and $(i^*\Mcal,i^!\Ncal)$ in $\Bcal$. Moreover, in this case, we have that $j_!j^*\Mcal=\Mcal\cap\Scal$ and $j_*j^*\Ncal=\Ncal\cap\Ccal$.
\end{enumerate}
\end{theorem}

The torsion pair in $\Tcal$ obtained in the theorem above is said to be \textbf{glued} from the torsion pair $(\Wcal,\Zcal)$ in $\Bcal$ and the torsion pair $(\Ucal,\Vcal)$ in $\Ycal$ along the given recollement.

The following well-known lemma will be useful in what follows.

\begin{lemma}\label{lem: restricts}
Let $\Ccal$ be a triangulated subcategory of a triangulated category $\Tcal$.  Let $(\Ucal, \Vcal)$ be a torsion pair in $\Tcal$ and suppose $\Ucal \subseteq \Ccal$, then $(\Ucal, \Vcal \cap \Ccal)$ is a torsion pair in $\Ccal$.  Similarly, if $\Vcal \subseteq \Ccal$, then $(\Ucal  \cap \Ccal, \Vcal)$ is a torsion pair in $\Ccal$.
\end{lemma}
\begin{proof}
Clearly the first two conditions of Definition \ref{Def: torsion pair} will restrict to $\Ccal$.  Moreover, if  $\Ucal$ is contained in $\Ccal$ then, for any $X$ in $\Ccal$, the first two terms of the triangle given by Definition \ref{Def: torsion pair} are contained in $\Ccal$.  Since $\Ccal$ is closed under extensions and shifts, it follows that the third term will be contained in $\Vcal \cap \Ccal$ and so $(\Ucal, \Vcal \cap \Ccal)$ is a torsion pair in $\Ccal$.  A similar argument yields the second claim.
\end{proof}

We finish with a result that states that any torsion pair with a definable torsion or torsionfree class can be obtained by glueing along a recollement of a trivial torsion pair with a torsion pair with certain nondegeneracy conditions. Moreover, we show that there is a natural operation generating a new torsion pair from an old one (see also \cite[Section 10]{PsSurvey} for an analogous construction). 

\begin{proposition}\label{Prop: constructing recoll}
Let $\Tcal$ be a compactly generated algebraic triangulated category and let $(\Ucal, \Vcal)$ be a torsion pair in $\Tcal$ with $\Ucal$ a definable subcategory. Consider the bireflective subcategory $\Bcal=\bigcap_{i\in \Z}\Ucal[i]={}^{\perp_\Z}\Vcal$ given by Corollary \ref{definable to bireflective} and the associated recollement 
\[\xymatrix@C=0.5cm{\Bcal \ar[rrr]^{i_*} &&& \Tcal \ar[rrr]^{j^*}  \ar @/_1.5pc/[lll]_{i^*}  \ar @/^1.5pc/[lll]_{i^!} &&& \Ycal\ar @/_1.5pc/[lll]_{j_!} \ar @/^1.5pc/[lll]_{j_*} }\]
where $i_*$ is the corresponding inclusion functor into $\Tcal$. Denote by $\Scal$ and $\Ccal$ the associated smashing and cosmashing subcategories, respectively.  Then the following statements hold. 
\begin{enumerate}
\item The torsion pair $(\Ucal,\Vcal)$ in $\Tcal$ is obtained by glueing:
\begin{itemize}
\item the left nondegenerate torsion pair $(j^*\Ucal,j^*\Vcal)$ which, under the equivalence between $\Ycal$ and $\Ccal$ induced by $j_*$, corresponds to the torsion pair $(\Ucal \cap \Ccal, \Vcal)$ in $\Ccal$;
\item the trivial torsion pair $(\Bcal,0)$ in $\Bcal$. 
\end{itemize}
\item The pair $(\Ucal \cap \Scal,\Bcal\ast \Vcal)$ is a left nondegenerate torsion pair in $\Tcal$ obtained by glueing\textcolor{blue}{:}
\begin{itemize}
\item the left nondegenerate torsion pair $(j^*\Ucal,j^*\Vcal)$;
\item the trivial torsion pair $(0,\Bcal)$ in $\Bcal$. 
\end{itemize}
\end{enumerate}
\end{proposition}
\begin{proof}
Firstly, note that indeed the equality $\bigcap_{i\in \Z}\Ucal[i] = {}^{\perp_\Z}\Vcal$ follows immediately from the fact that, for each $i$ in $\Z$, we have $\Ucal[i] = {}^\perp\Vcal[i]$. 

(1) First observe that, since $\Vcal$ is contained in $\Ccal$, we have that $j_*j^*\Vcal=\Vcal$. Therefore, by Theorem \ref{glueing}, the torsion pair $(\Ucal,\Vcal)$ is obtained by glueing $(j^*\Ucal,j^*\Vcal)$ and $(i^*\Ucal,i^!\Vcal)$. Moreover, by Lemma \ref{lem: restricts}, $(\Ucal \cap \Ccal, \Vcal)$ is a torsion pair in $\Ccal$ with torsionfree class $j_*j^*\Vcal=\Vcal$. Therefore, it must be the one corresponding under the equivalence $j_*$ between $\Ycal$ and $\Ccal$ to the torsion pair $(j^*\Ucal,j^*\Vcal)$ (i.e.~we get that $j_*j^*\Ucal=\Ucal\cap \Ccal$). Note also that $\cap_{i\in\Z}(\Ucal\cap\Ccal)[i]=\Bcal\cap\Ccal=0$ and, thus, $(j^*\Ucal,j^*\Vcal)$ is indeed left nondegenerate. Finally, observe that $i^!\Vcal=0$ since $\Ker(i^!)=\Ccal$ and $\Vcal$ is contained in $\Ccal$. Therefore we also have $i^*\Ucal=\Bcal$, completing the proof.

(2) We glue the torsion pairs  $(j^*\Ucal,j^*\Vcal)$ in $\Ycal$ and $(0,\Bcal)$ in $\Bcal$. By Theorem \ref{glueing}(1), the resulting torsion pair is $(j_!j^*\Ucal\ast 0,\Bcal\ast j_*j^*\Vcal)$. Since, as seen in Theorem \ref{glueing}(2), we have $j_!j^*\Ucal=\Ucal\cap\Scal$ and $j_*j^*\Vcal=\Vcal\cap\Ccal=\Vcal$, we get the desired torsion pair. Finally, observe that, as in (1), it is easy to see that the resulting torsion pair is left nondegenerate.
\end{proof}

Next we will consider the dual case of a definable torsionfree class. It follows from Proposition \ref{Cor: def torsf is TTF} that such a class is in fact a TTF class and, therefore, we obtain the following stronger result.

\begin{proposition}\label{Prop: constructing recoll dual}
Let $\Tcal$ be a compactly generated algebraic triangulated category and let $(\Ucal, \Vcal)$ be a torsion pair in $\Tcal$ with $\Vcal$ a definable subcategory. Consider the bireflective subcategory $\Bcal=\bigcap_{i\in \Z}\Vcal[i]={\Ucal}^{\perp_\Z}$ given by Corollary \ref{definable to bireflective} and the associated recollement 
\begin{equation}\label{def recoll}
\xymatrix@C=0.5cm{\Bcal \ar[rrr]^{i_*} &&& \Tcal \ar[rrr]^{j^*}  \ar @/_1.5pc/[lll]_{i^*}  \ar @/^1.5pc/[lll]_{i^!} &&& \Ycal\ar @/_1.5pc/[lll]_{j_!} \ar @/^1.5pc/[lll]_{j_*} }
\end{equation}
where $i_*$ is the corresponding inclusion functor into $\Tcal$. Denote by $\Scal$ and $\Ccal$ the associated smashing and cosmashing subcategories, respectively.  Then the following statements hold. 
\begin{enumerate}
\item The torsion pair $(\Ucal,\Vcal)$ in $\Tcal$ is obtained by glueing:
\begin{itemize}
\item the right nondegenerate torsion pair $(j^*\Ucal,j^*\Vcal)$ which, under the equivalence between $\Ycal$ and $\Scal$ induced by $j_!$, corresponds to the torsion pair $(\Ucal, \Vcal\cap\Scal)$ in $\Scal$;
\item the trivial torsion pair $(0,\Bcal)$ in $\Bcal$. 
\end{itemize}
\item The pair $(\Ucal\ast\Bcal,\Vcal\cap\Ccal)$ is a right nondegenerate torsion pair in $\Tcal$ obtained by glueing
\begin{itemize}
\item the right nondegenerate torsion pair $(j^*\Ucal,j^*\Vcal)$;
\item the trivial torsion pair $(\Bcal,0)$ in $\Bcal$.
\end{itemize}
\item The class $\Vcal$ can be expressed as $\Bcal\ast(\Vcal\cap\Ccal)$ and also as $(\Vcal\cap\Scal)\ast \Bcal$.
\end{enumerate}
\end{proposition}
\begin{proof}
Statements (1) and (2) are dual to Proposition \ref{Prop: constructing recoll}.  For statement (3) we use the fact that, by Proposition \ref{Cor: def torsf is TTF}, we have a torsion pair $(\Vcal,\Vcal^\perp)$. Indeed, this allows us to conclude, by Proposition \ref{Prop: constructing recoll}(1), that $j_*j^*\Vcal=\Vcal\cap\Ccal$. Using item (1) above, we then obtain that $\Vcal=\Bcal\ast(\Vcal\cap\Ccal)$ by Theorem \ref{glueing}. On the other hand, by applying Proposition \ref{Prop: constructing recoll}(1) to the torsion pair $(\Vcal,\Vcal^\perp)$ we have that it is obtained by glueing $(j^*\Vcal,j^*(\Vcal^\perp))$ and $(\Bcal,0)$.  By item (1) above, we have that $j_!j^*\Vcal = \Vcal \cap\Scal$ and so we conclude that $\Vcal = (\Vcal\cap\Scal)\ast \Bcal$ by Theorem \ref{glueing}.
\end{proof}

We end this paper by considering a special case of the above proposition where $\Vcal$ is cosuspended. This allows us to provide a structural description of t-structures with a definable coaisle.

\begin{proposition}
Let $\Tcal$ be a compactly generated algebraic triangulated category and let $(\Ucal, \Vcal)$ be a t-structure in $\Tcal$ with $\Vcal$ a definable subcategory. Consider the bireflective subcategory $\Bcal =\bigcap_{i\in \Z}\Vcal[i]$ given by Corollary \ref{definable to bireflective} and denote by $\Scal = {}^\perp\Bcal$ and $\Ccal=\Bcal^\perp$ the associated smashing and cosmashing subcategories, respectively.  Then there is a pure-injective object $C$ such that 
\begin{enumerate}
\item The smallest coaisle containing $C$ is $\Vcal_C=\Vcal\cap\Ccal$.
\item The definable coaisle $\Vcal$ can be expressed as $\Bcal\ast\Vcal_C$.
\item The heart $\Hcal$ of $(\Ucal,\Vcal)$ is equivalent to the heart of the t-structure $({}^{\perp_{\leq 0}}C,\Vcal_C)$.
\end{enumerate}
\end{proposition}
\begin{proof}
Since $\Vcal$ is definable we have that, by \cite[Theorem C]{SSV} and \cite[Theorem 3.11]{Laking}, the heart $\Hcal=\Ucal[-1]\cap\Vcal$ of $(\Ucal,\Vcal)$ is a Grothendieck abelian category. In particular, $\Hcal$ has an injective cogenerator $E$. Since $\Vcal$ is closed under coproducts, it follow that the cohomological functor $H^0\colon\Tcal\longrightarrow \Hcal$ commutes with coproducts (\cite[Lemma 3.3]{AMV3}). Using Brown representability as in \cite{NSZ}, we then conclude that there is an object $C$ representing the functor $\Hom_\Hcal(H^0(-),E)$ i.e., 
$$\Hom_\Tcal(-,C)\cong \Hom_\Hcal(H^0(-),E).$$
Moreover, by the proof of \cite[Theorem 3.6]{AMV3}, we have that $C$ is pure-injective.

(1) By Proposition \ref{Prop: constructing recoll dual}(2), we have a t-structure $(\Ucal\ast\Bcal,\Vcal\cap\Ccal)$ in $\Tcal$. We will show that this t-structure coincides with the t-structure $({}^{\perp_{\leq 0}}C,\Vcal_C)$ given by Corollary \ref{t-structure from pure-injective}. To show that $\Vcal_C$ is contained in $\Vcal\cap\Ccal$, it suffices to show that $C$ lies in $\Vcal\cap\Ccal$. By construction of the functor $H^0$ (see, for example, \cite{BBD}), it follows that $H^0(\Ucal)=0$ and $H^0(\Vcal[k])=0$ for all $k<0$. In particular, we have that $H^0(\Bcal)=0$ and, since $H^0$ is a cohomological functor, it follows that $H^0(\Ucal\ast\Bcal)=0$. This implies that, by construction of $C$,  we have that $C$ lies in $(\Ucal\ast\Bcal)^{\perp}=\Vcal\cap\Ccal$. 

For the reverse inclusion, let $Z$ be an object in ${}^{\perp_{\leq 0}}C$, and consider the triangle arising from the recollement (\ref{def recoll})
$$j_!j^*Z\longrightarrow Z\longrightarrow i_*i^!Z\longrightarrow j_!j^*Z[1].$$
Applying the cohomological functor $H^0$ to the triangle, we conclude that $H^0(j_!j^*Z[k])\cong H^0(Z[k])$ for all $k$ in $\Zbb$ because $H^0(\Bcal)=0$. Moreover, we have that $H^0(Z[k])=0$ for all $k\geq 0$ by definition of $C$ and choice of $Z$. Finally, note that by Proposition \ref{Prop: constructing recoll dual}(1), the functor $j_!$ identifies the right nondegenerate t-structure $(j^*\Ucal,j^*\Vcal)$ with $(\Ucal,\Vcal\cap\Scal)$ in $\Scal$.  Since $j_!j^*Z$ lies in $\Scal$, the nondegeneracy guarantees that $j_!j^*Z$ lies in $\Ucal$ if and only if $H^0(j_!j^*[k])=0$ for all $k\geq 0$ (see, for example, \cite{BBD}). Thus we have that, indeed, $j_!j^*Z$ lies in $\Ucal$ and $Z$ lies in $\Ucal\ast\Bcal$. This proves that ${}^{\perp_{\leq 0}}C\subseteq \Ucal\ast\Bcal$ and, hence, $\Vcal_C\supseteq \Vcal\cap\Ccal$.

(2) This follows by using (1) together with Proposition \ref{Prop: constructing recoll dual}(3).

(3) Recall that, by Proposition \ref{Prop: constructing recoll dual}(2), the t-structure $(\Ucal,\Vcal)$ can be obtained by glueing the t-structures $(0,\Bcal)$ and $(j^*\Ucal,j^*\Vcal)$. It then follows from \cite{BBD} (see also \cite[Section 4]{PsSurvey} for details) that the functor $j^*|_{\Hcal}\colon \Hcal\longrightarrow j^*\Hcal$ is a Serre quotient functor inducing a recollement of abelian categories and that, moreover, $\Ker(j^*|_{\Hcal})$ is the heart of the t-structure $(0,\Bcal)$ (which is zero). Therefore, 
$j^*|_{\Hcal}$ is indeed an equivalence of categories and since $(j_*\Ucal,j^*\Vcal)$ identifies with $(\Ucal\ast\Bcal,\Vcal\cap\Ccal)$ via the functor $j_*$ (see Proposition \ref{Prop: constructing recoll dual}(3)), the result follows by (1) above.
\end{proof}

\begin{remark}
Note that, in \cite{NSZ}, the pure-injective object $C$ obtained in the above proposition is known as a partial cosilting object.
\end{remark}

\end{document}